\documentclass[runningheads]{llncs}
\usepackage{graphicx}
\usepackage{pgfplots}
\usepackage{tikz}
\usepackage{amsmath, amssymb}
\usepackage{bm}
\usepackage{makecell}
\usepackage{mathtools}
\usepackage{subcaption}
\usepackage{hyperref}
\usetikzlibrary{matrix,calc,shapes}

\tikzset
{
  block/.style = {shape=rectangle, rounded corners,
                  draw, anchor=center, minimum height = 0em,
                  align=center, minimum width = 2em,
                  inner sep=1ex},
  operator/.style = {block, circle, inner sep = .3ex, minimum width = 0em, 
                     minimum height = 0em},
  summary/.style = {block, fill=black!10, rounded corners=0cm, minimum height = 
  0em,}
}

\newcommand{\edge}[2]{\draw (#1) edge node (TMP) {} (#2);}

\newcommand{\layername}[4]{\path (#1) -- node (TMP) {} (#2);
                           \draw node [left of=TMP, xshift=#4] {#3};}

\newcommand{\edgelabel}[1]{\draw node [right of=TMP, anchor=west, 
                           xshift=-2.2em] {#1};}
                         
\newcommand{\edgelabelleft}[1]{\draw node [left of=TMP, anchor=east, 
                               xshift=+2.2em] {#1};}

\newcommand{\skipedge}[3] {\draw (#1) -- +(#3,0) |-
                           node (TMP) {}
                           node [near start,left] {Id} (#2);}
                             
\newcommand{\emptyskipedge}[3] {\draw (#1) -- +(#3,0) |-
                           node (TMP) {}
                           (#2);}

\newcommand{\dx}{\partial_x}
\newcommand{\dt}{\partial_t}
\newcommand{\nor}[1]{||#1||_2}

\begin{document}
\title{Translating Numerical Concepts for PDEs\\ 
into Neural Architectures\thanks{This 
work has received funding from the European Research Council (ERC) 
under the European Union's Horizon 2020 research and innovation programme 
(grant agreement no. 741215, ERC Advanced Grant INCOVID).}}
\author{Tobias Alt \and
        Pascal Peter \and
        Joachim Weickert \and
        Karl Schrader} 
\authorrunning{T. Alt et al.}
\institute{Mathematical Image Analysis Group,
Faculty of Mathematics and Computer Science,\\
Campus E1.7, Saarland University, 66041 Saarbr\"ucken, Germany.\\
\email{\{alt, peter, weickert, schrader\}@mia.uni-saarland.de}}
\maketitle              % typeset the header of the contribution

\begin{abstract}
We investigate what can be learned from translating numerical algorithms
into neural networks. On the numerical side, we consider explicit,
accelerated explicit, and implicit schemes for a general higher order 
nonlinear diffusion equation in 1D, as well as linear multigrid 
methods. On the neural network side, we identify corresponding 
concepts in terms of residual networks (ResNets), recurrent networks,
and U-nets. These connections guarantee Euclidean stability 
of specific ResNets with a transposed convolution
layer structure in each block. We present three numerical justifications for 
skip connections: as time discretisations in explicit schemes, as 
extrapolation mechanisms for accelerating those methods, and as recurrent 
connections in fixed point solvers for implicit schemes. Last but not 
least, we also motivate uncommon design choices such as nonmonotone 
activation functions. Our findings give a numerical perspective 
on the success of modern neural network architectures, and they provide design 
criteria for stable networks.

%The abstract should briefly summarize the contents of the paper in

\end{abstract}

%%%%%%%%%%%%%%%%%%%%%%%%%%%%%%%%%%%%%%%%%%%%%%%%%%%%%%%%%%%%%%%%%%%%%%%%%%%

\section{Introduction}

The remarkable success of convolutional neural networks (CNNs) has
triggered many researchers to analyse their behaviour and to come up 
with mathematical foundations and stability guarantees.
One strategy consists of interpreting networks as approximations of
evolution equations; see e.g.~\cite{CP16,RH20,SPBD20}. Then training 
a network comes down to parameter identification of ordinary 
or partial differential equations (PDEs). This can be challenging, 
since it requires various model assumptions: 
Without additional smoothness assumptions, the models may become 
very complicated involving millions of parameters. 
Moreover, connecting a discrete network to a continuous evolution 
equation involves ambiguous limit assumptions: The same 
discrete model can approximate multiple evolution equations with 
different orders of consistency. 

\newpage 
\noindent
We address these problems by following two guiding principles: 
\begin{enumerate}
\item We refrain from the {\em analytic} strategy of translating a 
      complex neural network into a compact model, since it involves
      the discussed problems and only analyses how a network is,
      but not how it should be. Instead we pursue a {\em synthetic} 
      approach: We translate successful concepts into networks. 
      This is easier, and it allows to understand how a network 
      should be to guarantee desirable qualities such as stability and 
      efficiency. 
\item Our concepts of choice are numerical algorithms rather than continuous 
      models. This avoids ambiguities in the limit
      assumptions. Similar to neural architectures, 
      numerical algorithms can be applied to a multitude of models.
      We believe that the design principles of modern neural networks 
      realise a small but powerful set of numerical strategies as a basis of 
      their success.
      
\end{enumerate}
Thus, we want to justify key components of neural architectures and 
derive novel design principles by translating popular numerical 
algorithms into networks.

%------------------------------------------------------------------------

\subsubsection{Our Contributions.}
As an exemplary starting point and a basis for exploring different 
numerical algorithms, we consider a general evolution equation for 
higher order nonlinear diffusion in 1D. 

First we show that an explicit finite difference discretisation can
be interpreted as a residual network (ResNet) \cite{HZRS16}. 
This gives two central insights: The diffusion flux function determines 
the activation function of the ResNet, and the two convolutional filters 
follow a transposed structure. This motivates the use of nonmonotone 
activation functions and allows us to guarantee stability in the Euclidean 
norm for specific ResNets. Moreover, we identify the skip connections 
in the ResNet as discrete time derivatives. 

Additional interpretations of skip connections are obtained with 
alternative numerical methods based on fast semi-iterative (FSI) 
accelerations of explicit schemes \cite{HOWR16} and on fixed point 
algorithms for fully implicit schemes. We show that the latter 
ones can be regarded as recurrent neural networks \cite{Ho82}.

Since multigrid methods \cite{BHM00} are efficient numerical
methods for PDEs, it is worthwhile to analyse them as well.
We demonstrate that they have structural connections to U-nets 
\cite{RFB15}, shedding some light on their efficiency.

Our results do not only inspire general design criteria for neural networks 
as well as stable architectures. They also provide structural 
insights into ResNets, RNNs and U-nets from the perspective of numerical 
algorithms.

%------------------------------------------------------------------------

\subsubsection{Related Work.}
Our philosophy to translate numerics into neural networks is shared 
in \cite{LZLD18,OPF19}. Both works motivate 
additional skip connections in ResNets from multistep schemes for ordinary 
differential equations. Our paper provides alternative motivations
via time discretisations, acceleration via extrapolation, as well as fixed 
point schemes for implicit discretisations. 

The stability of ResNets is studied from a differential equations perspective 
in \cite{RDF20,RH20,ZS20}. Particularly, Ruthotto and 
Haber \cite{RH20} show stability in the Euclidean norm for a specific form of 
residual networks. However, they require the activation function to be 
monotone. In contrast to their approach, our theory allows nonmonotone 
activation functions.
 
Several works connect multigrid ideas and CNNs, e.g. to learn the restriction
and prolongation operators \cite{GGKY19}, or to couple feature channels for 
parameter reduction \cite{EERT20}. He and Xu \cite{HX19} present a CNN 
architecture implementing multigrid approaches, however without connecting 
it to a U-net. We on the other hand present a direct correspondence between 
a simple multigrid solver and a U-net.

%------------------------------------------------------------------------

\subsubsection{Organisation of the Paper.}
In Section \ref{sec:explicit}, we translate different numerical approximations 
for higher order diffusion into CNNs and analyse the resulting 
architectures. Covering multi-resolution approaches, we show that a multigrid 
solver can be cast into a U-net form in Section \ref{sec:multigrid}. Finally, 
we present our conclusions and an outlook on future work in Section 
\ref{sec:conclusions}.

%%%%%%%%%%%%%%%%%%%%%%%%%%%%%%%%%%%%%%%%%%%%%%%%%%%%%%%%%%%%%%%%%%%%%%%%%%%

\section{Networks from Algorithms for Evolution Equations}
\label{sec:explicit}

In this section, we start with generalised diffusion filters in 1D,
and we translate three numerical algorithms for them into neural 
network architectures. This gives novel insights into the value of 
skip connections, stable network design, and the potential of 
nonmonotone activation functions.

%------------------------------------------------------------------------

\subsection{Generalised Nonlinear Diffusion}

We consider a generalised higher order nonlinear diffusion model in 1D.
It creates filtered signals $u(x, t): (a,b) \times [0, \infty) \rightarrow 
\mathbb{R}$ from an initial signal $f(x)$ on a domain $(a,b) \subset 
\mathbb{R}$ according to the PDE
\begin{equation}\label{eq:diff}
  \dt u = -\mathcal{D}^* \! \left(g\!\left(|\mathcal{D} u|^2\right) \mathcal{D} 
                               u\right)
\end{equation}
which is the gradient flow that minimises the energy
% \begin{equation}
$E(u) = \int_a^b \Psi(|\mathcal{D} u|^2)\, dx$ 
% \end{equation}
with $g=\Psi'$.
We use a general differential operator $\mathcal{D} = \sum_{m=0}^M 
\alpha_m\partial_x^m$ and its adjoint version $\mathcal D^*=\sum_{m=0}^M (-1)^m 
\alpha_m\partial_x^m$, both consisting of weighted derivatives of up to order 
$M$. Thus, the corresponding PDE is of order $2M$.
The evolution is initialised at time $t=0$ by $u(x,0) = f(x)$, and we 
impose reflecting boundary conditions at $x=a$ and $x=b$.
Equation \eqref{eq:diff} creates gradually simplified versions of $f$.
The scalar \emph{diffusivity} function $g(s^2)$ controls the amount of 
smoothing depending on the local structure of the evolving signal. We consider 
diffusivities that are smooth, nonnegative, nonincreasing, and bounded from 
above.

Depending on the operator $\mathcal{D}$ and the choice of the diffusivity, the 
evolution describes different models. 
For $\mathcal{D} = \partial_x$, one obtains a 1D version of the nonlinear 
diffusion filter of Perona and Malik \cite{PM90}. For this model the 
exponential diffusivity $g(s^2) = \exp(-{s^2}/(2\lambda^2))$ inhibits 
smoothing near discontinuities where $|\dx u|$ exceeds a contrast 
parameter $\lambda$. This allows discontinuity-preserving smoothing.
A higher order choice of $\mathcal{D} = \partial_{x}^2$ yields a 
1D version of the fourth order PDE of You and Kaveh \cite{YK00}. 

%------------------------------------------------------------------------

\subsection{Residual Networks} 
Residual networks \cite{HZRS16} are a popular CNN architecture as they are 
easy to train, even for a high number of network layers. They consist of 
chained residual blocks. A residual block is made up of two convolutional 
layers with biases and nonlinear activation functions after each 
layer. Each block computes a discrete output $\bm u$ from an input 
$\bm f$ by
\begin{equation}\label{eq:residual_block}
  \bm u = \sigma_2 \! \left(
                      \bm f + \bm W_2 \, \sigma_1 \! 
                              \left(\bm W_1 \bm f + \bm b_1\right) 
                            + \bm b_2
                   \right),
\end{equation}
with discrete convolution matrices $\bm W_1, \bm W_2$, activation 
functions $\sigma_1, \sigma_2$ and bias vectors $\bm b_1, \bm b_2$. 
The main difference to feed-forward CNNs lies in the 
skip connection which adds the original input signal $\bm f$ to the result 
of the inner activation function $\sigma_1$. This facilitates training of very 
deep networks. 

%------------------------------------------------------------------------

\subsection{Expressing Explicit Schemes as Residual Networks}
In the following, we derive a direct correspondence between an explicit 
scheme for the generalised diffusion equation \eqref{eq:diff} and a ResNet. 
With the help of the \emph{flux function} $\Phi(s) = g(s^2) \, s$, we rewrite 
\eqref{eq:diff} as
\begin{equation}
  \dt u = - \mathcal{D}^* \! \left(\Phi(\mathcal{D} u)\right).
\end{equation}
We now discretise this equation with an explicit finite difference scheme.
To obtain discrete vectors $\bm u, \bm f \in \mathbb{R}^N$, we sample the 
continuous functions $u,f$ with distance $h$. We employ a forward 
difference with time step size $\tau$ for the time derivative. Moreover, we 
represent a discretisation of the operator $\mathcal D$ by a convolution matrix 
$\bm K$. Thus, the adjoint operator $\mathcal D^*$ is represented
by $\bm K^\top$. 

Starting with an initial signal $\bm u^0 = \bm f$, the evolving 
signal $\bm u^{k+1}$ at time step $k+1$ arises from the previous one by
\begin{equation}\label{eq:expl}
  \bm u^{k+1} = \bm u^k - \tau \bm K^\top \bm\Phi\!\left(\bm K \bm u^k\right).
\end{equation}
In this notation, the connection between an explicit diffusion step and a 
ResNet block becomes apparent:
\begin{theorem}[Diffusion-inspired ResNets]
A higher order diffusion step \eqref{eq:expl} is equivalent 
to a residual block \eqref{eq:residual_block} if
\begin{equation}\label{eq:translation}
  \sigma_1 = \tau \, \bm\Phi, \quad
  \sigma_2 = \textup{Id}, \quad
  \bm W_1 = \bm K, \quad
  \bm W_2 = -\bm K^\top,
\end{equation}
and the bias vectors $\bm{b}_1$, $\bm{b}_2$ are set to $\bm 0$. 
\end{theorem}
Interestingly, the inner activation function $\sigma_1$ corresponds 
to a scaled version of the flux function $\bm\Phi$. The effect of the 
skip connection in the residual block also becomes clear now: It is the central 
component to realise a time discretisation. We call a ResNet block of this form 
a \emph{diffusion block}. It is visualised in 
Figure~\ref{fig:diffusion_block}(a). Graph nodes contain the current state 
of the signal, while edges describe operations which are applied to proceed 
from one node to the next. 

We observe that the convolution matrices satisfy $\bm W_2 = - \bm W_1^\top$. 
This is a direct consequence of the gradient flow structure of the diffusion 
process. In the following, we prove stability and well-posedness for this 
specific form of ResNets.

%..........................................................................

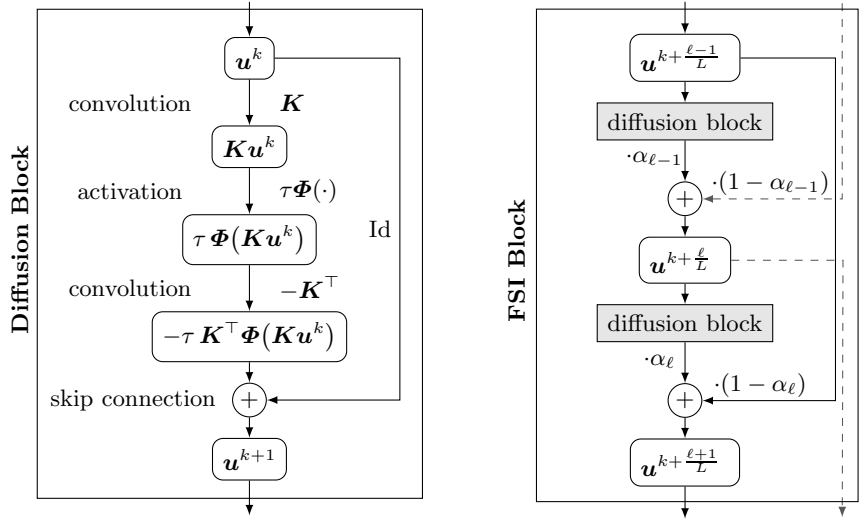
\begin{figure}[t]
  \centering
  \begin{subfigure}{.45\textwidth}
  \begin{tikzpicture}[-latex]
  \matrix (m)
  [
    matrix of nodes,
    column 1/.style = {nodes={block}}
  ]
  {
    |(input)|          % |[node type] (block name)|
    $\bm u^k$%         % expression within block (be careful withtrailing
                       % whitespace oustide math mode)
    \\[4.5ex]            % [distance to next row]
    |(fwd)|
    $\bm K \bm u^k$%
    \\[4.5ex]
    |(flux)|
    $\tau \, \bm\Phi\!\left(\bm K \bm u^k\right)$%
    \\[4.5ex]
    |(bwd)|
    $-\tau \, \bm K^\top \bm\Phi\!\left(\bm K \bm u^k\right)$%
    \\[2ex]
    |[operator] (add)|
    $+$%
    \\[2ex]
    |(output)|         
    $\bm u^{k+1}$%
    \\
  };

  % variables for shifts
    \def\nameshift{-1.8em}                    
    \def\skipconnshift{2.0}                     % shift for skipconnection to 
                                              % get around the largest block in 
                                              % the graph

    \layername{input}{fwd}{convolution}{\nameshift}   
    \edge{input}{fwd}
    \edgelabel{$\bm K$}
    
    \layername{fwd}{flux}{activation}{\nameshift}      
    \edge{fwd}{flux}     
    \edgelabel{$\tau \bm\Phi(\cdot)$}           
                                              
    \layername{flux}{bwd}{convolution}{\nameshift}                             
    \edge{flux}{bwd}          
    \edgelabel{$-\bm K^\top$}

    \node[left of=add, xshift=-4ex] {skip connection};
    \edge{bwd}{add}                             

    \edge{add}{output}
  
    \skipedge{input}{add}{\skipconnshift}   

  % manual input/output arrows
    \node[above of=input, yshift=-1ex] (TMP) {};
    \draw (TMP) -- (input);

    \node[below of=output, yshift=1ex] (TMP) {};
    \draw (output) -- (TMP);

  % frame box
    \node[left of=input, xshift=\nameshift-3.8em, yshift=2em] (topleft) {};
    \node[right of=output, xshift=4em, yshift=-1.6em] (bottomright) {};
    \draw (topleft) rectangle (bottomright);

  % rotated text
    \node [left of=flux, rotate=90, anchor=north, yshift=7em]
          {\textbf{Diffusion Block}};
\end{tikzpicture}
  \end{subfigure}
  \hfill
  \begin{subfigure}{.45\textwidth}
  \vspace{-3.5mm}
  \begin{tikzpicture}[-latex]
  \matrix (m)
  [
    matrix of nodes,
    column 1/.style = {nodes={block}}
  ]
  {
    |(input)|          % |[node type] (block name)|
    $\bm u^{k + \frac{\ell - 1}{L}}$ % expression within block 
    \\[2ex]            % [distance to next row]
    |[summary] (diff1)|
    diffusion block%
    \\[4ex]
    |[operator] (add1)|
    $+$%
    \\[2ex]
    |(middle)|
    $\bm u^{k + \frac{\ell}{L}}$
    \\[2ex]
    |[summary] (diff2)|
    diffusion block%
    \\[4ex]
    |[operator] (add2)|
    $+$%
    \\[2ex]
    |(output)|
    $\bm u^{k + \frac{\ell + 1}{L}}$
    \\
  };

  % variables for shifts
    \def\nameshift{-1em}                    
    \def\skipconnshift{2.0}                     % shift for skipconnection to 
                                              % get around the largest block in 
                                              % the graph
  
    \edge{input}{diff1}
    \edge{diff1}{add1}
    \edge{add1}{middle}
    \edge{middle}{diff2}
    \edge{diff2}{add2}
    \edge{add2}{output}
    
    \emptyskipedge{input}{add2}{\skipconnshift}
    \node[right of=add2, yshift=1.5ex] {$\cdot (1-\alpha_{\ell})$};
    \node[above of=add2, xshift=-2.5ex, yshift=-3.5ex] {$\cdot \alpha_{\ell}$};
    
    \node[right of=output, xshift=8ex, yshift=-6.2ex] (TMP) {} ;
    \draw[thin, dashed, black!65] (middle) -- +(2.1,0) -- (TMP);
    \node[thin, dashed, right of=add1, yshift=1.5ex, xshift=1ex] 
    {$\cdot (1-\alpha_{\ell-1})$};
    \node[thin, dashed, above of=add1, xshift=-3ex, yshift=-3.5ex] 
    {$\cdot \alpha_{\ell-1}$};
        
    \node[right of=input, xshift=8ex, yshift=6.2ex] (TMP) {} ;
    \draw[thin, dashed, black!65] (TMP) |- (add1);

  % manual input/output arrows
    \node[above of=input, yshift=-1ex] (TMP) {};
    \draw (TMP) -- (input);

    \node[below of=output, yshift=1ex] (TMP) {};
    \draw (output) -- (TMP);

  % frame box
    \node[left of=input, xshift=-3em, yshift=2em] (topleft) {};
    \node[right of=output, xshift=4em, yshift=-1.6em] (bottomright) {};
    \draw (topleft) rectangle (bottomright);

  % rotated text
    \node [left of=middle, rotate=90, anchor=north, yshift=4.5em]
          {\textbf{FSI Block}};
\end{tikzpicture}  
  \end{subfigure}
  \vspace{-7mm}
 \caption{{(a) Left:} Diffusion block for an explicit    
          diffusion step \eqref{eq:expl} with flux function $\bm\Phi$, time 
          step size $\tau$, and a discrete derivative operator $\bm K$.
         {(b) Right:} FSI block.
        \label{fig:diffusion_block}}
\end{figure}

%--------------------------------------------------------------------------

\subsection{Criteria for Well-posed and Stable Residual Networks}
Now we are able to transfer stability \cite{DWB09} and 
well-posedness \cite{We97} results for diffusion to a residual network 
consisting of diffusion blocks. We show Euclidean stability, which 
states that the Euclidean norm of the signal is nonincreasing in each 
iteration, i.e. $\nor{\bm u^{k+1}} \leq \nor{\bm u^k}$. Well-posedness 
guarantees that the network output is a continuous function of the input data.

\begin{theorem}[Euclidean Stability of ResNets with Diffusion Blocks]
\label{theo:stab}
  Consider a residual network chaining any number of diffusion blocks 
  \eqref{eq:expl} with convolutions represented by a convolution matrix 
  $\bm K$ and activation function $\tau \bm \Phi$. Moreover, assume that
  the activation function arises from a diffusion flux function 
  $\Phi(s) = g(s^2) \, s$ with finite Lipschitz constant $L$. Then the 
  residual network is well-posed and stable in the Euclidean norm if 
  %
  %\begin{equation}\label{eq:bound}
  $  \tau \leq 2 \left(L \nor{\bm K}^2\right)^{-1}.$
  %\end{equation} 
  %
  Here, $||\cdot||_2$ denotes the spectral norm which is induced by the 
  Euclidean norm. 
\end{theorem}
\begin{proof}
  We first notice that since $\Phi(s) = g(s^2) \, s$, 
  applying the flux function leads to a rescaling with a diagonal matrix $\bm 
  G(\bm u^k)$ with $g((\bm K \bm u^k)^2_i)$ as $i$-th diagonal element. 
  Therefore, we can write \eqref{eq:expl} as
  \begin{equation}
    \bm u^{k+1} = \left(\bm I - \tau \bm K^\top \bm G(\bm 
      u^k) \bm K \right) \bm u^k.
  \end{equation}
  At this point, well-posedness follows directly from the continuity of the 
  operator $\bm I - \tau \bm K^\top \bm G(\bm u^k) \bm K$, as  
  the diffusivity $g$ is assumed to be smooth \cite{We97}.
    
  We now show that the time step size restriction guarantees 
  that the eigenvalues of the operator always lie in the interval $[-1, 1]$.
  As the spectral norm is sub-multiplicative, we can estimate the eigenvalues 
  of $\bm K^\top \bm G(\bm u^k) \bm K$ for each matrix separately.
  Since $g$ is nonnegative, the diagonal matrix $\bm G$ is positive 
  semi-definite. The maximal eigenvalue of $\bm G$ is the given by the supremum 
  of $g$. As $g$ is non-increasing and bounded, this value is bounded by the
  Lipschitz constant $L$ of $\Phi$.
  %
%  \begin{equation}
%    \frac{\partial}{\partial s} \Phi(s) = g(s^2)
%    + 2 s^2 g^\prime(s^2) \leq L.
%  \end{equation}
  %
  Thus, the eigenvalues of $\bm K^\top \bm G(\bm u^k) \bm K$ lie in the 
  interval $[0, \tau L \nor{\bm K}^2]$. Consequently, the operator
  $ \bm I - \tau \bm K^\top \bm G(\bm u^k) \bm K$ has 
  eigenvalues in $[1 - \tau L \nor{\bm K}^2, 1]$, and the condition 
  $ 1 - \tau L \nor{\bm K}^2 \geq -1$
  leads to the bound  $\tau \leq 2 \left(L \nor{\bm K}^2\right)^{-1}.$
  \qed
\end{proof}

\subsubsection{How General is this Result?}
Theorem \ref{theo:stab} is of fairly general nature and applies to a 
broad class of ResNets.
The fact that $\bm K$ represents a
discrete differential operator is no restriction on the convolution, since
any convolution kernel can be seen as a discretisation of a suitable 
differential operator $\mathcal{D} = \sum_{m=0}^M \alpha_m\partial_x^m$.

Interestingly, our proof does not require the matrix $\bm K$ to have a 
convolution structure: It can be any arbitrary matrix. This even includes 
neural networks beyond CNNs, since the weights within a layer may 
differ from node to node.
\begin{center}
\setlength{\fboxsep}{2mm}
\setlength{\fboxrule}{0.4mm}
\fbox{
\begin{minipage}{0.7\linewidth}
{\bf The key requirement for network stability is the\\
transposed convolution structure $\bm W_2 = -\bm W_1^T$.} 
\end{minipage}
}
\end{center}
While this requirement is not fulfilled by the original ResNet 
\cite{HZRS16}, several works employ the transposed structure 
\cite{CP16,RH20,ZS20} as it is justified from a PDE perspective, requires less 
parameters, and provides stability guarantees.

In contrast to Ruthotto and Haber \cite{RH20}, our stability result does not 
require activation functions to be monotone. Let us now see that widely used 
diffusivities naturally lead to nonmonotone activation functions.

%------------------------------------------------------------------------

\subsection{Nonmonotone Activation Functions}
The connection between diffusivity $g(s^2)$ and activation function
$\sigma(s) = \tau \Phi(s)$ with the diffusion flux $\Phi(s) = g(s^2) \, s$ 
revitalises an old idea of neural network design \cite{FMNP93,MR94}. 
As an example, we translate the exponential Perona--Malik diffusivity 
$g(s^2)=\exp\bigl(-\frac{s^2}{2\lambda^2}\bigr)$ 
into its corresponding activation 
$\sigma(s)=\tau s \exp\bigl(-\frac{s^2}{2\lambda^2}\bigr)$. 
Interestingly, this activation function is {\em antisymmetric} and 
{\em nonmonotone.} 

Antisymmetry is very natural in the diffusion case with
$\mathcal{D} = \sum_{m=1}^M \alpha_m\partial_x^m$, where the argument of the 
flux function consists of signal derivatives. It reflects the invariance 
axiom that signal negation and filtering are commutative.
Nonmonotone flux functions were considered somewhat problematic for continuous 
diffusion PDEs. However, it has been shown that their discretisations are 
well-posed \cite{WB97}, in spite of the fact that they may act contrast 
enhancing.

The concept of a nonmonotone activation function is unusual in the CNN world. 
Although there have been a few early proposals in the neural network literature 
arguing in favour of nonmonotone activations \cite{FMNP93,MR94}, they are 
rarely used in modern CNNs. In practice, CNNs often fix the activation to 
simple functions such as the rectified linear unit (ReLU).
From a PDE perspective, this appears restrictive. The diffusion interpretation 
suggests that activation functions should be learned in the same manner as 
convolution weights and biases. In practice, this hardly happens apart from a 
few notable exceptions such as \cite{CP16,GWMC13,OMLM18}. As nonmonotone 
flux functions outperform monotone ones in the diffusion setting, it appears 
promising to incorporate them into CNNs. For more examples of 
diffusion-inspired activation functions, we refer to \cite{AWP20}. 

%------------------------------------------------------------------------

\subsection{FSI Schemes and Additional Skip Connections}\label{sec:fsi}

In the following, we show that an acceleration strategy of the 
explicit scheme induces a natural modification for the skip connections of the 
corresponding ResNet architecture.
To speed up explicit schemes, Hafner et al.~\cite{HOWR16} proposed 
\emph{fast semi-iterative} (FSI) schemes. They perform a cycle of
extrapolated explicit steps. For our diffusion scheme \eqref{eq:expl},
an FSI acceleration with cycle length $L$ reads
\begin{equation}\label{eq:fsi}
  \bm u^{k+\frac{\ell+1}{L}} = \alpha_\ell 
  \left(\bm I - \tau \bm K^\top \bm \Phi\!\left(\bm K 
  \bm u^{k+\frac{\ell}{L}}\right)\right) + 
  \left(1-\alpha_\ell\right)\bm u^{k+\frac{\ell-1}{L}}
\end{equation}
with $\ell=0,\dots,L\!-\!1$ and extrapolation weights 
$\alpha_\ell \coloneqq (4\ell+2)/(2\ell+3)$. 
One formally initialises with $\bm{u}^{k-\frac{1}{L}} \coloneqq \bm u^{k}$. 
This cycle realises a super time step of size
$\frac{L(L+1)}{3}\tau$. Thus, with one cycle involving $L$ explicit steps, 
one reaches a super step size of $\mathcal{O}(L^2)$ rather than
$\mathcal{O}(L)$. This explains its 
remarkable efficiency \cite{HOWR16}.

We see that FSI extrapolates the diffusion result at time step 
$k+\frac{\ell}{L}$ with the previous time step $k+\frac{\ell-1}{L}$ 
and the weight $\alpha_\ell$. This can be realised with a small change
in the original diffusion block from Figure~\ref{fig:diffusion_block}(a)
by adding an additional skip connection. The two skip connections are weighted 
by $\alpha_\ell$ and $(1-\alpha_\ell)$, respectively. This gives the 
architecture in Figure~\ref{fig:diffusion_block}(b).

We observe a different benefit of skip connections: Additional and more general 
skip connections constitute a whole class of acceleration strategies, which is 
in line with observations in the CNN literature; see e.g. \cite{HLMW17,LZLD18}.

%--------------------------------------------------------------------------

\subsection{Implicit Schemes and Recurrent Neural Networks}
So far, we have connected variants of explicit schemes to ResNets. However, 
implicit discretisations are another important class of solvers. We now 
show that such a discretisation of our diffusion equation leads 
to a recurrent neural network (RNN). RNNs are classical neural network 
architectures; see e.g.~\cite{Ho82}. 
The fully implicit discretisation of \eqref{eq:diff} is given by
\begin{equation}
  \bm u^{k+1} = \bm 
  u^k - \tau \bm K^\top \bm \Phi\!\left(\bm K \bm u^{k+1}\right).
\end{equation}
We solve the resulting nonlinear system of equations by $L$ fixed point 
iterations:
\begin{equation}
  \bm u^{k+\frac{\ell+1}{L}} = \bm u^k - \tau \bm K^\top \bm \Phi\! 
  \left(\bm K \bm u^{k+\frac{\ell}{L}}\right),
\end{equation}
where $\ell=0,\dots,L\!-\!1$, and where we assume that $\tau$ is sufficiently
small to yield a contraction mapping.
For $L=1$, we obtain the explicit scheme \eqref{eq:expl} with its 
ResNet interpretation. For larger $L$, however, different skip connections  
arise. They connect the layer at time step $k$ with all subsequent 
layers at steps $k+\frac{\ell}{L}$ with $\ell=0,\dots,L\!-\!1$. This feedback 
can be seen as an RNN architecture. 

In the context of variational models, Chen and Pock \cite{CP16} have obtained 
a similar architecture. However, they explicitly supplement the diffusion
process with an additional reaction term which results from the data term
of the energy. Our feedback term is a pure numerical phenomenon of the
fixed point solver.

We see that skip connections can implement a number of successful numerical 
concepts: 
forward difference approximations of the time derivative in explicit schemes,
extrapolation steps to accelerate them e.g.~via FSI, and
recurrent connections within fixed point solvers for implicit schemes.

%%%%%%%%%%%%%%%%%%%%%%%%%%%%%%%%%%%%%%%%%%%%%%%%%%%%%%%%%%%%%%%%%%%%%%%%%%%

\section{Multigrid Solvers and U-nets}\label{sec:multigrid}

Multigrid methods \cite{BHM00} are very efficient numerical 
strategies for solving PDE-based problems. Neural 
networks, on the other hand, have benefitted from multiscale ideas 
as well, as can be seen e.g.~from the high popularity of U-nets~\cite{RFB15}.
In this section we shed some light on their structural connections. For 
simplicity we restrict ourselves to a linear multigrid setting with 
two levels.

%-----------------------------------------------------------------------------

\subsection{U-net Architectures}
The U-net \cite{RFB15} has proven useful in applications such as segmentation 
\cite{RFB15} or pose estimation \cite{NYD16}, where features on multiple scales 
need to be extracted 

As its name suggests, the U-net has a symmetric shape: On the left half of the 
architecture, convolutions extract features while repeated downsampling 
operators reduce the resolution. On the right half, features  
are successively upsampled, combined and convolved, starting with the coarsest 
resolution. The original U-net \cite{RFB15} combines features by concatenation, 
while other works such as \cite{NYD16} use addition. In the following, we focus 
on the latter design choice.

For our purposes, it is sufficient to consider a U-net with only two 
resolutions and a constant number of channels. We use superscripts $h$ and $H$ 
to denote computations on the fine and coarse grid, respectively. The following 
six steps capture the essential structure of such a U-net:
\begin{enumerate}
\item One applies a number of CNN layers to the input $\bm f^h$, 
yielding a modified signal $\bm{\tilde f}^h$. We denote this general operation 
by a function $\bm C_1^h(\cdot)$. 
\item To provide a coarse input $\bm f^H = \bm R^{h\rightarrow H} \bm{\tilde 
f}^h$ to the next level, a restriction operator $\bm R^{h\rightarrow H}$ brings 
the modified signal $\bm{\tilde f}^h$ to a coarse resolution $H$. For example, 
the restriction can consist of an averaging or max-pooling.
\item On the coarse grid, the downsampled signal $\bm f^H$ is again modified by 
a series of layers to obtain $\bm{\tilde f}^H = \bm C^H(\bm f^H)$. 
\item One upsamples the coarse result $\bm{\tilde f}^H$ with a 
prolongation operator $\bm P^{H\rightarrow h}$.
\item On the fine grid, one adds the modified fine grid signal $\bm{\tilde 
f}^h$ and the upsampled one $\bm P^{H\rightarrow h}\bm f^H$ and obtains
$\bm{\tilde f}_{\text{new}}^h$.
\item Lastly, applying more layers $\bm C_2^h(\cdot)$ yields
the final solution $\bm{\hat{f}} = \bm C_2^h(\bm{\tilde f}_{\text{new}}^h)$.
\end{enumerate}
Figure~\ref{fig:multigrid_and_unet}(a) visualises this architecture. In the 
following, we express a multigrid V-cycle in this form by utilising multiple 
network channels.

%............................................................................
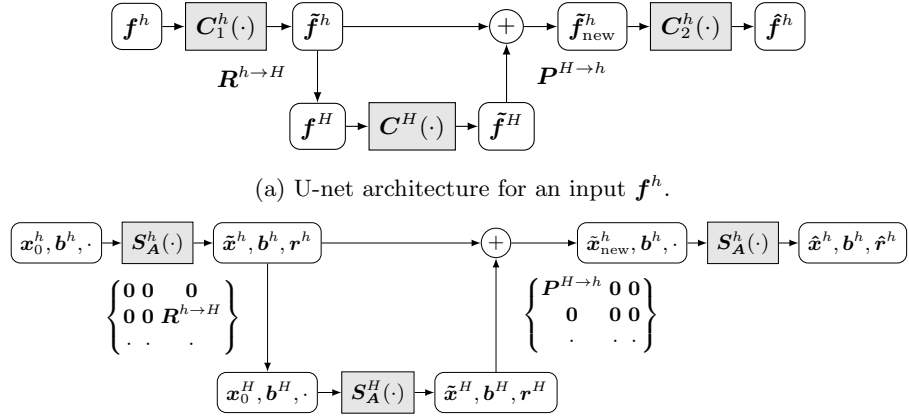
\begin{figure}[t]
\begin{subfigure}{\textwidth}
\centering
\begin{tikzpicture}[-latex]
  \matrix (m)
  [
    matrix of nodes,
    column sep=3mm,
    column 1/.style = {nodes={block}},
    column 2/.style = {nodes={block}},
    column 3/.style = {nodes={block}},
    column 4/.style = {nodes={block}},
    column 5/.style = {nodes={block}},
    column 6/.style = {nodes={block}},
    column 7/.style = {nodes={block}},
    column 8/.style = {nodes={block}}
  ]
  {
    % row 1 
        % col 1
        |(input)|          % |[node type] (block name)|
        $\bm f^h$%        
    % row 1
    &   % col 2
        |[summary] (conv1)|
        $\bm C_1^h(\cdot)$%
    % row 1
    &   % col 3
        |(conv1out)|
        $\bm{\tilde f}^h$%
    % row 1
    &   % col 4
        %
        %
    % row 1
    &   % col 5
        |[operator] (add)|
        $+$%
    % row 1
    &   % col 6
        |(conv2in)|
        $\bm{\tilde f}_{\text{new}}^h$%
    % row 1
    &   % col 7
        |[summary] (conv2)|
        $\bm C_2^h(\cdot)$%
    % row 1
    &   % col 8
        |(output)|
        $\bm{\hat{f}}^h$%
    \\[5ex]
    % row 2 
        % col 1
        %
        %
    % row 2
    &   % col 2
        %
        %
    % row 2
    &   % col 3
        |(down)|
        $\bm f^H$%
    % row 2
    &   % col 4
        |[summary] (downconv)|
        $\bm C^H(\cdot)$%
    % row 2
    &   % col 5
        |(downconvout)|
        $\bm{\tilde f}^H$%
    % row 2
    &   % col 6
        %
        %
    % row 2
    &   % col 7
        %
        %
    % row 2
    &   % col 8
    \\
  };
  
  \edge{input}{conv1}
  \edge{conv1}{conv1out}
  \edge{conv1out}{add}
  \edge{add}{conv2in}
  \edge{conv2in}{conv2}
  \edge{conv2}{output}
  
  \edge{conv1out}{down}
  \edgelabelleft{$\bm R^{h\rightarrow H}$}
  \edge{down}{downconv}
  \edge{downconv}{downconvout}
  \edge{downconvout}{add}
  \edgelabel{$\bm P^{H \rightarrow h}$}

\end{tikzpicture}
\caption{U-net architecture for an input $\bm f^h$.}
\vspace{1mm}
\end{subfigure}
\begin{subfigure}{\textwidth}
\centering
\resizebox{\textwidth}{!}{\begin{tikzpicture}[-latex]
  \matrix (m)
  [
    matrix of nodes,
    column sep=3mm,
    column 1/.style = {nodes={block}},
    column 2/.style = {nodes={block}},
    column 3/.style = {nodes={block}},
    column 4/.style = {nodes={block}},
    column 5/.style = {nodes={block}},
    column 6/.style = {nodes={block}},
    column 7/.style = {nodes={block}},
    column 8/.style = {nodes={block}},
  ]
  {
    % row 1 
        % col 1
        |(input)|          % |[node type] (block name)|
        $\bm x_0^h, \bm b^h, \cdot$%        
    % row 1
    &   % col 2
        |[summary] (solv1)|
        $\bm S^h_{\bm A}(\cdot)$%
    % row 1
    &   % col 3
        |(solv1out)|
        $\bm{\tilde x}^h, \bm b^h, \bm r^h$%
    % row 1
    &   % col 4
        %
        %
    % row 1
    &   % col 7
        |[operator] (add)|
        $+$%
    % row 1
    &   % col 9
        |(solv2in)|
        $\tilde{\bm{x}}_{\text{new}}^h, \bm b^h, 
        \cdot$%
    % row 1
    &   % col 9
        |[summary] (solv2)|
        $\bm S^h_{\bm A}(\cdot)$%
    % row 1
    &   % col 10
        |(output)|
        $\bm{\hat x}^h, \bm b^h, \bm{\hat{r}}^h$%
    \\[12ex]
    % row 2 
        % col 1
        %
        %
    % row 2
    &   % col 2
        %
        %
    % row 2
    &   % col 3
        |(down)|
        $\bm x_0^H, \bm b^H, \cdot$%
    % row 2
    &   % col 4
        |[summary] (downsolv)|
        $\bm S^H_{\bm A}(\cdot)$%
    % row 2
    &   % col 5
        |(downsolvout)|
        $\bm{\tilde x}^H, \bm b^H, \bm r^H$%
    % row 2
    &   % col 7
        %
        %
    % row 2
    &   % col 8
        %
        %
    % row 2
    &   % col 8
    \\
  };
  
  \edge{input}{solv1}
  \edge{solv1}{solv1out}
  \edge{solv1out}{add}
  \edge{add}{solv2in}
  \edge{solv2in}{solv2}
  \edge{solv2}{output}
  
  \edge{solv1out}{down}
  \edgelabelleft{$\begin{Bmatrix}
  \bm 0 & \bm 0 &\bm 0 \\
  \bm 0 & \bm 0 & \bm R^{h \rightarrow H} \\
  \cdot & \cdot &\cdot
  \end{Bmatrix}$}
  \edge{down}{downsolv}
  \edge{downsolv}{downsolvout}
  \edge{downsolvout}{add}
  \edgelabel{$\begin{Bmatrix}
    \bm P^{H \rightarrow h} & \bm 0 &\bm 0 \\
    \bm 0 & \bm 0 &\bm 0 \\
    \cdot & \cdot & \cdot
    \end{Bmatrix}$}

\end{tikzpicture}}
\caption{Two-level V-Cycle in the form of a U-net utilising three-channel 
signals containing the iteration variable $\bm x$, a right hand side $\bm 
b$, and the residual $\bm r$.}
\end{subfigure}
\caption{Architectures for a general U-net and a multigrid V-cycle.
\label{fig:multigrid_and_unet}}
\end{figure}

%------------------------------------------------------------------------------

\subsection{Expressing a Multigrid V-cycle within a U-net}

Multigrid methods \cite{BHM00} allow for 
the efficient solution of equation systems that result from the numerical 
approximation of PDEs. For simplicity we consider a linear system of equations
given by $\bm A \bm x = \bm b$. For classical iterative solvers such as the
Jacobi or the Gauss--Seidel method, one observes that low-frequent error
components are attenuated only very slowly. Hence, their convergence is
slow. Multigrid methods transfer the low-frequent error components to a
coarser scale, where the iterative solvers work more efficiently. The coarse 
scale solution is then used to correct the fine scale approximation. 

To connect multigrid ideas to U-nets, we consider a V-cycle on two levels 
with grid sizes $h$ and $H$ for the fine and coarse grid.
For our U-net, we use three channels. They contain
the iteration variable $\bm x$ of the solver, the right hand side $\bm b$ of 
the equation system, and the current residual $\bm r$. Even though we do not 
always need all channels, we keep the channel number constant for simplicity. 
A two-level V-cycle solves the linear system by repeating the following steps:
\begin{enumerate}
\item The inputs are a fine grid initialisation $\bm x_0^h = \bm 0$
and the given right hand side $\bm b^h$. The residual at this point is ignored, 
as it is not relevant to the solver input. We 
assume that we are given a solver $\bm S^h_{\bm A}(\cdot)$ 
for the operator $\bm A^h$. It produces a three-channel 
signal containing an approximate solution $\bm{\tilde x}^h$, the right hand 
side $\bm b^h$, and a residual $\bm r^h = \bm b^h - \bm A^h \bm{\tilde x}^h$. 
\item While the true error $\bm e^h$ of the approximation is unknown, the 
residual $\bm r^h$ can be computed. This leads to the residual equation $\bm 
A^h \bm e^h = \bm r^h$ which can be solved efficiently on a coarser grid. To 
this end, one uses a restriction operator $\bm R^{h\rightarrow H}$. As the 
downsampling is now explicitly concerned with three channels, the corresponding 
operator in the CNN is a $3\times3$ block matrix. The coarse initialisation 
$\bm x_0^H = \bm 0$ does not require any information from the fine scale. 
Crucially, the new right hand side $\bm b^H$ is the downsampled fine residual, 
i.e. $\bm b^H = \bm R^{h \rightarrow H} \bm r^h$. Lastly, an input residual is 
not required for the coarse solver. Thus, one obtains the coarse grid residual 
equation $\bm A^H \bm x^H = \bm b^H$.
\item The coarse grid solver $\bm S^H_{\bm A}(\cdot)$ now solves the 
residual equation. It outputs a coarse approximation $\bm{\tilde x}^H$ to the 
residual error, the right hand side $\bm b^H$, and a new coarse residual $\bm 
r^H$. The latter two would be required if one wants to add another level to the 
cycle.
\item In the upsampling step, we prepare the coarse scale 
outputs for the following addition. Similar to the downsampling, we upsample 
only the coarse error approximation $\bm{\tilde x}^H$ by a prolongation 
operator $\bm P^{H \rightarrow h}$. The coarse right hand side $\bm b^H$ is set 
to $\bm 0$ as to not interfere with the fine right hand side. 
\item On the fine grid, one adds the three signal channels. The initial fine 
grid approximation is updated with the upsampled error on the coarse grid by 
$\tilde{\bm{x}}_{\text{new}}^h \coloneqq \bm{\tilde x}^h  + \bm P^{H 
\rightarrow h}\bm{\tilde x}^H$. As we have applied the prolongation only to the 
coarse solution, the fine grid right hand side $\bm b^h$ is propagated.
\item Another instance of the fine grid solver $\bm S^h_{\bm A}(\cdot)$ takes 
the corrected solution $\bm x_\text{new}^h$ and the original right 
hand side $\bm b^h$, yielding a new approximation $\bm{\hat{x}}^h$.
\end{enumerate} 
We visualise this architecture in Figure~\ref{fig:multigrid_and_unet}(b). 
Restriction and prolongation operators are applied only to 
certain channels of the solver output instead of all channels. In the 
downsampling phase, the restriction is applied to the residual, while in the 
upsampling phase, it is applied to the approximated error. This enables the 
coarse solver to work on the residual equation instead of only a coarse version 
of the original equation, which is the crucial idea of multigrid methods.
The architecture utilises yet another form of skip connection: The fine scale 
approximation is corrected by adding an upsampled error approximation.

Our two-level setting can be generalised to more levels. 
Deeper V-cycles are constructed by stacking 
the two-level V-cycle recursively, and so-called 
W-cycles are built by concatenating two V-cycles. On the CNN side, this leads 
to U-nets with more levels, as well as concatenations thereof. This idea is 
also used in practice: Successful U-nets work on multiple resolutions 
\cite{RFB15}, and so-called stacked hourglass models \cite{NYD16} arise by 
concatenating multiple V-cycle architectures. It shows that 
multigrid architectures share essential structural properties with U-nets.

%%%%%%%%%%%%%%%%%%%%%%%%%%%%%%%%%%%%%%%%%%%%%%%%%%%%%%%%%%%%%%%%%%%%%%%%%%%

\section{Conclusions}\label{sec:conclusions}

Our paper is based on the philosophy of regarding a trained neural network
as a numerical algorithm. To substantiate this claim, we have translated 
a number of efficient numerical algorithms for PDEs into popular building
blocks for network architectures. Apart from a few notable exceptions 
such as \cite{LZLD18}, this strategy has rarely been pursued in its full 
consequence. We have shown that valuable structural insights can be 
gained from such a direct translation, and we have derived systematic 
design principles for well-founded network components.

More specifically, we have shown the value of skip connections from three 
different numerical perspectives: as time discretisations in explicit 
schemes, as extrapolation terms to increase their efficiency, and as 
recurrent connections in implicit schemes with fixed point structure. 
By connecting multigrid methods to U-nets, we provide a basis 
for explaining for their remarkable efficiency. Numerical schemes for
generalised diffusion processes suggest that nonmonotone 
activation functions are permissible and can be advantageous. Last but not 
least, we have seen that a ResNet block with a transposed structure
of both convolution layers can guarantee Euclidean
stability in a simple and elegant way.

Our contributions can serve as a blueprint for translating a larger class 
of successful numerical concepts for PDEs to CNNs. This is part of our
ongoing work. It is our hope that this will lead to a closer connection
of both worlds and to hybrid methods that unite the stability and efficiency
of modern numerical algorithms with the performance of neural networks.
 
\medskip
\noindent\textbf{Acknowlegdements.} We thank Matthias Augustin and Michael 
Ertel for fruitful discussions and feedback on our manuscript.

 \bibliographystyle{splncs04}
 \bibliography{refs}

\end{document}